\documentclass[a4paper,11pt]{article}

\usepackage[T1]{fontenc}
\usepackage[english]{babel}
\usepackage{dsfont}
\usepackage{amsfonts}
\usepackage{amsmath}
\usepackage{amssymb,latexsym}
\usepackage{mathrsfs}
\usepackage{enumitem}
\usepackage{ stmaryrd }
\usepackage{amsthm}
\usepackage{authblk}
\usepackage{csquotes}
\usepackage{subcaption}
\usepackage{hyperref}
\usepackage{cleveref}
\usepackage{comment}
\usepackage{graphicx} 
\graphicspath{ {./images/} }

\usepackage[left=2.5cm, right=2.5cm, top=3cm]{geometry}

\newtheorem{lem}{Lemma}[section]
\newtheorem{prop}[lem]{Proposition}
\newtheorem{rem}[lem]{Remark}
\newtheorem{defi}[lem]{Definition}
\newtheorem{theo}[lem]{Theorem}

\renewcommand{\P}{\mathbb{P}}
\newcommand{\N}{\mathbb{N}}
\newcommand{\E}{\mathbb{E}}
\newcommand{\Z}{\mathbb{Z}}
\newcommand{\R}{\mathbb{R}}

\newcommand{\eps}{\varepsilon}
\newcommand{\A}{\mathcal{A}}
\newcommand{\D}{\mathcal{D}}
\title{\textsc{\Huge{Geodesics and Wandering Exponents in Brochette First-Passage Percolation}}}
\author{Maxime Marivain}
\affil{CMAP, \'Ecole Polytechnique, CNRS, Route de
Saclay, 91128 Palaiseau Cedex, France\\\emph{maxime.marivain@polytechnique.edu}}
\date{\today}
\begin{document}
\maketitle
\begin{abstract}
We study geodesics in the Brochette first-passage percolation model, where edges on the same axis-parallel line share a common random passage time, inducing long-range dependence. We focus on the maximal transversal deviation $H_n$ of geodesics from the origin to $ne_1$.
We prove existence of geodesics under mild assumptions and establish the order of magnitude of $H_n$ depending on the behavior of the passage-time distribution near its infimum. These results yield explicit wandering exponents in this dependent setting.
\end{abstract}
\section{Introduction}

First-passage percolation is a probabilistic model originally proposed by Hammersley and Welsh in \cite{Hammersley} to study the speed of  propagation of a fluid through a porous medium. (See \cite{Survey} for a survey on this topic.)
We consider the lattice $\Z^d$ with edge set $\E^d:=\{\{x;y\}:\lVert x-y\rVert_1=1\}$. To each edge $e\in \E^d$, we assign a non-negative random variable $\tau_e$ , called the passage time of $e$.  We denote by 
$F$ the common distribution function of the passage times.
A (finite or infinite) sequence of edges such that consecutive edges share exactly one endpoint is called a path. The passage time of a path $\Gamma$ is defined by
$$
T(\Gamma)=\sum_{e\in\Gamma}\tau_e.
$$
If $x\in \R^d$, we denote by $x'$ the unique vertex in $\Z^d$ such that 
$x\in x'+ [0;1)^d$. For $x,y\in\R^d$, we then define the passage time between $x$ and $y$ as
$$
T(x,y)=\inf_{\Gamma} T(\Gamma),
$$
where the infimum is taken over all paths $\Gamma$ connecting $x'$ and $y'$.
Finally, for $t\ge 0$, we define
$$
B_t :=\{x\in \R^d:T(0,x)\le t\},
$$
which represents the set of points that can be reached from the origin within time $t$. In the classical first-passage percolation model, the family 
$(\tau_e)_{e\in\E^d}$  is assumed to be i.i.d.. 
In \cite{moi}, we introduced the Brochette first-passage percolation model as a new model of first-passage percolation with long-range dependence, and derived classical asymptotic results for $T(x,y)$ and $B_t$. In this paper, we are interested in the height of geodesics in this new model.
 \subsection{Brochette first-passage percolation}
 First, we recall the definition of Brochette first-passage percolation as defined in \cite{moi}. We keep all the definitions of the previous paragraph and we now specify how passage times are associated to each edge.
 In this model, each axis-parallel line is assigned an i.i.d. positive random variable, shared by all of its edges. Furthermore, the passage times associated with edges on distinct lines remain independent.
 Let us now define our model more formally.
 An integer line is the set of points in $\Z^{d}$ that lie on the same axis-parallel line of $\R^{d}$. We denote by $\Delta$ the set of these lines.
  \begin{defi}[Brochette first-passage percolation]\label{hyp:model}
 Let $(\tau_{\delta})_{\delta\in\Delta}$ be a family of non-negative i.i.d. random variables with common cumulative distribution function $F$. In the Brochette first-passage percolation, 
 edges whose endpoints belong to the same integer line $\delta\in\Delta$ are assigned the same passage time.
 In other words, if $\delta$ is the unique integer line to which the edge $e$ belongs, we set $\tau_e=\tau_\delta$.
  \end{defi}
  Throughout the paper, $a\ge0$ denotes the infimum of the support of the passage-time distribution, and $F$ denotes its cumulative distribution function.
  In particular we will use that:
 \begin{equation}\label{hypo:a}
     \forall \eps>0,\; F( a+\eps)>0 \;\;and\;\; F(a-\eps)=0.
 \end{equation}
 To state and prove our results, we recall a theorem from a previous work. It asserts that the asymptotic shape of $B_t$ is deterministic and depends only on the infimum $a$ of the support of the passage-time distribution.
\begin{theo}\label{th:forme}\cite{moi}
   Consider Brochette first-passage percolation on $\Z^d$. Assume that   
    \begin{equation}\label{condition}
    \E[\min(\tau_{1}^{d},\dots,\tau_{d}^{d})]<+\infty, 
    \end{equation}
    where $(\tau_i)_{1\le i\le d}$ is i.i.d. with distribution $F$ and that $a>0$ verifies \eqref{hypo:a}. Then for all $\eta>0$ we have:
    $$
    \P(\frac{1-\eta}{a}\Diamond\subset \frac{B_{t}}{t}\subset \frac{1}{a}\Diamond \text{ for all large }t)=1,
    $$
    where $\Diamond$ is the $\ell_1$-ball of $\R^d$ of radius $1$.
\end{theo}
For $d\ge2$, we denote by $(e_1,\dots,e_d)$ the canonical basis of $\R^d$. If $x\in\R^d$ and $i\in\llbracket 1;d\rrbracket$, $x_i$ will denote the coordinate of $x$ according to $\R e_i$.
For $x,y\in\Z^d$, we say that a path $\gamma$ from $x$ to $y$ is a geodesic from $x$ to $y$ if $T(\gamma)=T(x,y)$.
In this article, we study the deviation of geodesics from the straight line. To this end, we introduce the notion of the height of a path.
Let $\gamma$ be a path starting from $0$. We define
    $$
H(\gamma):=\max_{\substack{x\in e\\ e\in\gamma}}\lVert x-x_1e_1\rVert_1,
    $$
which we call the height of $\gamma$. 
Let $n\in\N$. We define the random variable given by the maximal height of a
geodesic from $0$ to $ne_1$
    $$
      H_n:=\max_{\Gamma_n} H(\Gamma_n),
    $$
where the maximum is taken over all the geodesics from $0$ to $ne_1$.
\section{The results}
We first discuss existence of geodesics, according to the distribution $F$. 
\begin{prop}\label{th:geopos}
    Let $d\ge2$. Consider Brochette first-passage percolation on $\Z^d$ with passage-time distribution $F$. Assume that one of these two assumptions is satisfied:
    \begin{enumerate}[label=(\roman*)]
        \item $a>0$,
    
        \item $a=0$ and $F(t)\overset{t\to0}{\sim}ct^\beta$, with $\beta\ge1$ and $c>0$.
    \end{enumerate}
     Then the following holds
    $$
     \P(\text{for all } x,y\in\Z^d \text{, a geodesic exists from } x \text{ to } y)=1.
    $$
    
\end{prop}
\begin{rem}
Proposition \ref{th:geopos} states that geodesics almost surely exist for $a=0$ when $F(t)\overset{t\to0}{\sim}ct^\beta$ with $\beta\ge1$. We prove it by showing that in that case, the time needed to go to an infinite distance from the origin is infinite. When $\beta<1$, this time is finite. Thus we don't know if the result still holds when $\beta<1$.
\end{rem}
In the classical first-passage percolation model, it is conjectured that geodesics between any two vertices exist almost surely for every passage time distribution. This was proved in dimension $d=2$ by Wierman and Reh in \cite{wierman}.
In dimension $d \ge 3$, Kesten proved in \cite{Kesten} the almost sure existence of geodesics under the assumption that $F(0) < p_c(d)$, where $p_c(d)$ denotes the critical probability for Bernoulli bond percolation on $\Z^d$.
Later, Zhang showed in \cite{zhang} that almost sure existence also holds when $F(0) > p_c(d)$, provided that $\E[\tau_e] < +\infty$.
The critical case $F(0) = p_c(d)$ remains open. Our Proposition~\ref{th:geopos} establishes the almost sure existence of geodesics when $a>0$, and also in the case $a=0$ under additional assumptions on the distribution $F$.
In these cases, we can focus on their deviation from the straight line. Thus,  we bound the height $H_n$, depending on the distribution of the passage times. More formally we have the two following results.
\begin{theo}[Height of geodesics when $a=0$]\label{th:geo0}
Assume that $a=0$, $d=2$ and $F(t)\overset{t\to0}{\sim}ct^\beta$ where $\beta\ge1$ and $c>0$. Then for all $\eta>0$, there exists $\eps>0$ such that
\begin{equation}\label{lowbo0}
\liminf_{n\to+\infty}{\P\left( H_n\ge \eps n \right)}\ge 1-\eta.
\end{equation}
Furthermore, if the support of the distribution of passage times is bounded and if $\beta>1$,  for all $\eta>0$, there exists $K>0$ such that
\begin{equation}\label{upbo0}
\liminf_{n\to+\infty}{\P\left( H_n\le Kn\right)}\ge 1-\eta.
\end{equation}
\end{theo}
\begin{rem}

    In the second statement of Theorem \ref{th:geo0}, we assume that the support of the passage-time distribution is bounded in order to control $\E[T(0,ne_1)]$. Indeed, we use the fact that, according to Theorem $1.7$ in \cite{moi}, if $\beta>1$, then there exist constants $A_1,A_2>0$ such that $A_1n^{1-\frac{1}{\beta}}\le \E[T(0,ne_1)]\le A_2n^{1-\frac{1}{\beta}}$. However, we believe that the result should still hold under a weaker assumption. Moreover, even though we proved the existence of geodesics for $\beta=1$, determining $H_n$ in this case appears to be more challenging. We leave this question open.
    
\end{rem}
\begin{theo}[Height of geodesics when $a>0$]\label{th:geoa}
Assume that $a>0$, $d\ge2$ and $F(t)\overset{t\to a}{\sim}c(t-a)^{\beta}$ where $\beta>0$ and $c>0$.  Assume also that \eqref{condition} is verified. Then for all $\eta>0$, there exist $\eps,K>0$ such that
$$
\liminf_{n\to+\infty}{\P\left( \eps n^{\frac{\beta}{\beta+d-1}} \le H_n\le Kn^{\frac{\beta}{\beta+d-1}} \right)}\ge 1-\eta.
$$
\end{theo}
The height of geodesics is a central open problem in classical first-passage percolation. It is conjectured that there exists a constant $\xi(d)$, called the wandering exponent, such that for any passage-time distribution, $H_n\approx n^{\xi(d)}$.
The precise and most relevant interpretation of $\approx$ in the classical setting is still a matter of ongoing investigation. Licea, Newman, and Piza \cite{licea} propose four definitions of $\xi(d)$, which they conjecture to be equivalent, and derive lower bounds for these quantities. Newman and Piza also prove in \cite{newman} that $\xi(d)\le\frac{3}{4}$ for all $d\ge2$.
In dimension two, the wandering exponent $\xi(2)$ is expected to equal $\frac{2}{3}$.
For the Brochette model, Theorems \ref{th:geo0} and \ref{th:geoa} show that, under a suitable definition of the wandering exponent and appropriate conditions on $F$, one has $\xi(2)=1$ when $a=0$, and $\xi(d)=\frac{\beta}{\beta+d-1}$ when $a>0$.
\begin{rem}
    If $Y$ is such that $F_Y(t)\overset{t\to a}{\sim}(t-a)^\beta$, then $X:=c^{-\frac{1}{\beta}}(Y-a)+a$ is such that $F_{X}(t)\overset{t\to a}{\sim}c(t-a)^\beta$. Thus, we will consider in all the proofs that $F(t)\overset{t\to a}{\sim}(t-a)^\beta$.
\end{rem}
\begin{rem}
Since multiple geodesics may exist, one could instead define
$$
\tilde{H}_n:=\min_{\Gamma_n}H(\Gamma_n),
$$
and use it instead of $H_n$. However, Theorems $\ref{th:geo0}$ and $\ref{th:geoa}$ would still hold with the same proofs.
\end{rem}
\section{Preliminaries: existence of geodesics}
In this section, we prove the existence of geodesics in several cases. We first define the passage time to infinity:
$$
    \rho:=\rho(F)=\lim_{n\to+\infty}T\left(0,\partial \mathcal{B}(n) \right),
$$
where $\mathcal{B}(n)$ denotes the $\ell_1$-ball of $\Z^d$ of radius $n$ and the boundary of a set $S$ is denoted by $\partial S:=\{ x\in S: \exists y\in\Z^d\setminus S \text{ such that }\lVert x-y \rVert_1=1\}$. Now we recall a result from \cite{Survey} on the link between $\rho$ and the existence of geodesics in the classical model. 
\begin{lem}[Proposition 4.4 in \cite{Survey}]\label{prop:rhoinf}
    Assume that $\rho=+\infty$ for some configuration. Then there exists a geodesic from $x$ to $y$ for all $x,y\in\Z^d$.
\end{lem}
\begin{proof}
   The proof of Proposition 4.4 in \cite{Survey} does not contain any probabilistic argument. Thus, Lemma \ref{prop:rhoinf} can be proved in the exact same way in the Brochette model.
\end{proof}
\begin{proof}[Proof of Proposition \ref{th:geopos}]
If $(i)$ holds we notice that since $a>0$, any path $\Gamma_n$ of length $n$ is such that $T\left(\Gamma_n\right)\ge an$. Thus, for all $n\in \N$, $\rho>an\to_{n\to+\infty}+\infty$. Then, the result follows from Lemma \ref{prop:rhoinf}.\par
We now detail the proof in the case where $(ii)$ holds, namely when $a=0$ and $F(t)\overset{t\to 0}{\sim}ct^\beta$. Let us show that $\rho=+\infty$. 
Let $\Gamma_n$ be a path starting at $0$ that goes out of $\mathcal{B}(n)$. Then $\Gamma_n$ must cross an edge in every $\mathcal{B}(r)$, for all $0\le r\le n$. Thus, if we denote by $N_i$ the infimum of the passage times of edges in $\mathcal{B}(i)$   we get
    $$
T(\Gamma_n)\ge \sum_{i=1}^{n}N_i.
    $$
    Moreover, for all $i\in\N$ we have
    $$
N_{i+1}=\min(N_i,V_{i+1}),
    $$
    where $V_{i+1}$ is the infimum of the $2d(d-1)$ independent new passage times in $\mathcal{B}(i+1)$. Indeed, due to the specific structure of the Brochette model, the only new passage times are those associated with integer lines passing through the points $\pm ne_i$ for $i\in \llbracket 1,d \rrbracket$, and distinct from $\Z e_i$. There are $2d$ such points and $d-1$ such integer lines through each point. We define the moment when $(N_i)_{i\ge 1}$ decreases by induction
    $$
T_0:=0 \text{ and if }\; T_j \; \text{ is defined }\; T_{j+1}:=\inf\{ i\ge T_j+1: N_{i+1}<N_i \}.
    $$
Let $\mathcal{F}_n$ denote the $\sigma$-algebra generated by $(N_i)_{i\le n}$. Conditionally to $\mathcal{F}_{T_{j-1}}$, $T_j-T_{j-1}$ follows a geometric distribution of parameter $p_j$ such that
$$
    p_j=\P\left(V_j\le N_{T_{j-1}}|N_{T_{j-1}}\right)
       =1-\P\left( \tau> N_{T_{j-1}}|N_{T_{j-1}}\right)^{2d(d-1)}
       =1-\left(1-F\left(N_{T_{j-1}}\right)\right)^{2d(d-1)}.
$$
Thus we obtain
\begin{align*}
    \P\left( \left(T_j-T_{j-1}\right)N_{T_{j-1}}>1|\mathcal{F}_{T_{j-1}}\right)&=\left( 1-p_j\right)^{\lfloor\frac{1}{N_{T_{j-1}}}\rfloor}\\
    &=\left( 1-  F\left(N_{T_{j-1}}\right)\right)^{2d(d-1)\lfloor\frac{1}{N_{T_{j-1}}}\rfloor}\\
    &=\exp\left(2d(d-1)\lfloor\frac{1}{N_{T_{j-1}}}\rfloor\log\left( 1-F\left(N_{T_{j-1}}\right)\right)\right).
\end{align*}
Now, since $N_{T_{j-1}}$ converges almost surely to $0$, $\lim_{j\to+\infty}F\left(N_{T_{j-1}}\right)=0$ and $\frac{1}{N_{T_{j-1}}}\overset{j\to+\infty}{\sim}\lfloor\frac{1}{N_{T_{j-1}}}\rfloor$.
Thus we have
$$
\P\left( \left(T_j-T_{j-1}\right)N_{T_{j-1}}>1|\mathcal{F}_{T_{j-1}}\right)\overset{j\to+\infty}{=}\exp\left( -2d(d-1)F\left(N_{T_{j-1}}\right)\lfloor\frac{1}{N_{T_{j-1}}}\rfloor+o\left(\frac{F\left(N_{T_{j-1}}\right)}{N_{T_{j-1}}}\right)\right).
$$
Moreover, since $F(t)\overset{t\to0}{\sim}t^\beta$, we get
$
F\left(N_{T_{j-1}}\right)\overset{j\to+\infty}{\sim}N_{T_{j-1}}^\beta.
$
It follows that, since $\beta\ge 1$,
$$
\lim_{j\to+\infty}F\left(N_{T_{j-1}}\right)\lfloor\frac{1}{N_{T_{j-1}}}\rfloor= \left\{
    \begin{array}{ll}
        0 & \mbox{if } \beta>1 \\
        1 & \mbox{if }  \beta=1
    \end{array}
\right.,
$$
and as a result
$$
\lim_{j\to +\infty}\P\left( \left(T_j-T_{j-1}\right)N_{T_{j-1}}>1|\mathcal{F}_{T_{j-1}}\right)=\exp\left(-2d(d-1)\mathds{1}_{\beta=1}\right) \qquad a.s..
$$
We obtain that for all $n\in\N$
$$
\rho>\sum_{i=1}^{T_n}N_i\ge\sum_{j=1}^{n-1}N_{T_j}\left( T_j-T_{j-1}\right).
$$
Thus, by letting $n$ go to $+\infty$, this yields
$$
\rho\ge \sum_{j=1}^{+\infty}N_{T_j}\left( T_j-T_{j-1}\right) \ge \mathrm{Card}\left(\{ i\ge1: N_{T_i}\left( T_i-T_{i-1}\right)\ge 1 \}\right).
$$
Finally, we know by the conditional version of the Borel-Cantelli Lemma that this cardinal is infinite. Hence the result using Lemma \ref{prop:rhoinf}.
\end{proof}
\begin{rem}
    When $\beta<1$, we can show that $\rho<+\infty$ almost surely. Therefore, we cannot deduce the existence of geodesics.
\end{rem}
\section{The wandering exponent}
In this section, we prove Theorem \ref{th:geo0} and \ref{th:geoa}. To this end, we employ the same strategy several times. We first consider a path with minimal passage time among those constrained to lie above or below a given height. We then construct a second path that attains the height specified in the theorem and show that it is asymptotically more efficient.
First, we recall a result due to Fisher, Tippett and Gnedenko on the limiting law for the minimum of $i.i.d.$ random variables.
\begin{theo}[Fisher, Tippett and Gnedenko]\label{weibull}
    Let $X_1,\dots,X_n$ be $i.i.d.$ random variables with common distribution function $F$ such that $F(t)\overset{t\to a}{\sim}(t-a)^{\beta}$, where $a$ is the infimum of the support of $X_1$ and $\beta>0$. Then if we denote by $\sigma_n:=\min{\{X_1,\dots,X_n}\}$, we have
    $$
n^{\frac{1}{\beta}}(\sigma_n-a)\overset{(d)}{\to}_{n\to+\infty} Y,
    $$
    where $Y$ follows a Weibull distribution of parameters $\beta$ and $1$.
\end{theo}
\subsection{The case $a=0$}
This subsection is devoted to the proof of Theorem \eqref{th:geo0}. Thus, we assume that $a=0$, $d=2$, and $F(t)\overset{t\to0}{\sim}t^\beta$ with $\beta\ge1$ for the lower bound and $\beta>1$ for the upper bound. We first prove the lower bound \ref{lowbo0}: when $n$ goes to $+\infty$, $H_n\ge\eps n$ with high probability. Then we prove the upper bound \ref{upbo0}: when $n$ goes to $+\infty$, $H_n\le Kn$ with high probability.
\begin{proof}[Proof of Theorem \eqref{th:geo0}]
    \textbf{Lower bound.} If $x,y\in\Z^d$ belong to the same integer line $\delta$, we denote by $\llbracket x,y \rrbracket$:=$[x,y]\cap \Z^d.$
    Let $0<\eps<1$, and let $\Gamma_n^1$ be a path from $0$ to $ne_1$ with minimal passage time among all paths $\Gamma$ such that $H(\Gamma)\le \eps n$.
    To prove the lower bound we construct a path $\Gamma_n^2$ such that $H(\Gamma_n^2)>\eps n$
    and we show that this path is asymptotically better than $\Gamma_n^1$. Since, $\{T(\Gamma_n^2)-T(\Gamma_n^1)\le 0\}\subset\{H_n>\eps n\}$ it is sufficient to get a lower bound on $\P(T(\Gamma_n^2)-T(\Gamma_n^1)\le 0)$.
    First, we define 
\begin{itemize}
    \item $\rho_n(\eps):=\min\{\tau_{y+\Z e_1}:y\in\llbracket \lfloor -\eps n \rfloor e_2, \lfloor\eps n\rfloor e_2\rrbracket\}$.
    \item  $\sigma_n(\eps):=\min\{\tau_{y+\Z e_1}:y\in\llbracket (\lfloor\eps n\rfloor+1)e_2, \lfloor\sqrt{\eps} n\rfloor e_2\rrbracket\}.$
    \item $\mu_n:=\min\{\tau_{x+\Z e_2}:x\in\llbracket 0, \lfloor \frac{n}{3}\rfloor e_1\rrbracket\} $
    \item $\nu_n:=\min\{\tau_{x+\Z e_2}:x\in\llbracket  \lceil \frac{2n}{3}\rceil e_1,ne_1\rrbracket\}$.
\end{itemize}
Thus, $\rho_n(\varepsilon)$ (resp. $\sigma_n(\varepsilon)$) denotes the minimal horizontal passage time over integer lines with ordinate between $\lfloor -\varepsilon n \rfloor$ and $\lfloor \varepsilon n \rfloor$ (resp. between $\lfloor \varepsilon n \rfloor + 1$ and $\lfloor \sqrt{\varepsilon} n \rfloor$). Moreover, $\mu_n$ (resp. $\nu_n$) denotes the minimal vertical passage time over integer lines with abscissa between $0$ and $\lfloor \frac{n}{3} \rfloor$ (resp. between $\lceil \frac{2n}{3} \rceil$ and $n$). We refer the reader to Figure \ref{Figure hauteur0} for an illustration of the construction of $\Gamma_n^2$, which we now describe.
Let $x_{1,n}\in\llbracket 0, \lfloor \frac{n}{3}\rfloor e_1\rrbracket$  (Resp. $x_{2,n}\in \llbracket  \lceil \frac{2n}{3}\rceil e_1,ne_1\rrbracket$) be a vertex of $\Z e_1$ such that $\tau_{x_{1,n}+\Z e_2}=\mu_n$ (Resp. $\tau_{x_{2,n}+\Z e_2}=\nu_n$). Let $y_n$ be the vertex of $\Z e_2$ such that $\tau_{y_n+\Z e_1}=\sigma_n(\eps)$. 
We define $\gamma_{1,n}$ as the path that follows $\Gamma_n^1$ until it crosses $x_{1,n}+\Z e_2$ for the first time. Then $\gamma_{2,n}$ is the path that goes straight from the end point of $\gamma_{1,n}$ to $x_{1,n}+y_n$ then straight to $x_{2,n}+y_n$ and straight again to the last intersection vertex between $\Gamma_n^1$ and $x_{2,n}+\Z e_2$. Finally, $\gamma_{3,n}$ is the path that follows $\Gamma_n^1$ from the end of $\gamma_{2,n}$ to $ne_1$ and $\Gamma_n^2$ is the concatenation of $\gamma_{1,n}$, $\gamma_{2,n}$ and $\gamma_{3,n}$. 
\begin{figure}
\centerline{\includegraphics[scale=0.4]{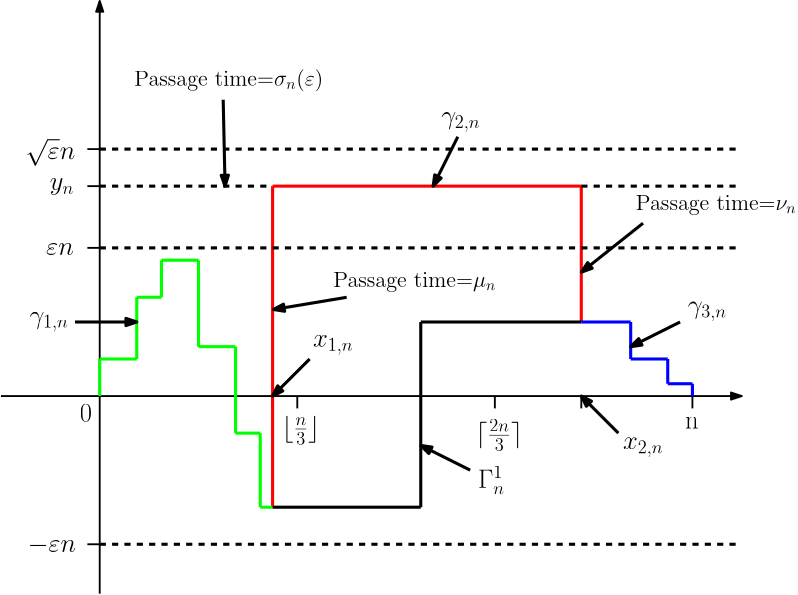}}~
\caption{The initial path $\Gamma_n^1$ is the concatenation of $\gamma_{1,n}$ (in green), the black path and $\gamma_{3,n}$ (in blue). The path $\Gamma_n^2$ is the concatenation of $\gamma_{1,n}$ (in green), $\gamma_{2,n}$ (in red), and $\gamma_{3,n}$ (in blue).}
\label{Figure hauteur0}
\end{figure}
If $\sigma_n(\eps)<\rho_n(\eps)$, this construction yields
    $$
    T(\Gamma_n^{2})-T(\Gamma_n^1)\le \frac{n}{3}(\sigma_n(\eps)-\rho_n(\eps))
    +(\eps +\sqrt{\eps}) n(\mu_n+\nu_n).
    $$ 
    Thus we have
    \begin{align*}
\P(H_n>\eps n)&\ge\P(T(\Gamma_n^2)-T(\Gamma_n^1)\le 0)\\
                  &\ge \P(T(\Gamma_n^2)-T(\Gamma_n^1)\le 0 \text{ and } \sigma_n(\eps)<\rho_n(\eps))\\
                  &\ge\P\left(\frac{n}{3}(\sigma_n(\eps)-\rho_n(\eps))+(\eps +\sqrt{\eps}) n(\mu_n+\nu_n)\le0\right)\\
                  &=\P\left(n^{\frac{1}{\beta}}(\rho_n(\eps)-\sigma_n(\eps))\ge 3(\eps+\sqrt{\varepsilon})n^{\frac{1}{\beta}}(\mu_n+\nu_n)\right).
    \end{align*}
 Now, using Theorem \ref{weibull}, $n^{\frac{1}{\beta}}\rho_n(\eps)$ and $n^{\frac{1}{\beta}}\sigma_n(\eps)$ converge in distribution to $Y_\varepsilon$ and $Y'_\varepsilon$ respectively such that $Y_\eps$ and $Y'_\eps$ are independent and $\varepsilon^{\frac{1}{\beta}}Y_\varepsilon$ and $(\sqrt{\varepsilon}-\varepsilon)^{\frac{1}{\beta}}Y'_\varepsilon$ follow a Weibull distribution of parameters $\beta$ and 1.
Moreover, since the lines involved in the definitions of $\rho_n(\eps)$ and $\sigma_n(\eps)$ are distinct, the random variables $n^{\frac{1}{\beta}}\rho_n(\eps)$ and $n^{\frac{1}{\beta}}\sigma_n(\eps)$ are independent for all $n\in\N$. Thus $n^{\frac{1}{\beta}}(\rho_n(\eps)-\sigma_n(\eps))$ converges in distribution to $Y_\varepsilon-Y'_\varepsilon.$ We show by the same reasoning that $n^{\frac{1}{\beta}}\left(\mu_n+\nu_n\right)$ converges in distribution to a positive random variable $Z$, such that $Z$ is independent from $Y_\eps-Y'_\eps$. Now, $\P(Y_\varepsilon> Y'_\varepsilon)\to_{\varepsilon\to0}1$. Hence 
$$
\P\left(n^{\frac{1}{\beta}}(\rho_n(\eps)-\sigma_n(\eps))\ge 3(\eps+\sqrt{\varepsilon})n^{\frac{1}{\beta}}(\mu_n+\nu_n)\right)\to_{n\to+\infty}\P\left( Y_\eps-Y'_\eps\ge 3\left(\eps+\sqrt{\eps}\right)Z\right)\to_{\eps\to0}1,
$$
which yields $\lim_{\eps\to0}\liminf_{n\to+\infty}{\P\left( H_n\ge \eps n \right)}=1$.\par
\medskip
\textbf{Upper bound.} 
Let $K>0$ and $\Gamma_n$ a geodesic from $0$ to $ne_1$. On the event $\{H(\Gamma_n)>Kn \}$,
$\Gamma_n$ must go through at least $Kn$ edges in $\llbracket \lfloor- Kn\rfloor;\lfloor Kn \rfloor \rrbracket^2$.  We denote by $\varsigma_n(K)$ the infimum of the passage times of the edges having both of their extremities in $\llbracket \lfloor- Kn\rfloor;\lfloor Kn \rfloor \rrbracket^2$. (See Figure \ref{Figure hauteur01}.)
\begin{figure}
\centerline{\includegraphics[scale=0.33]{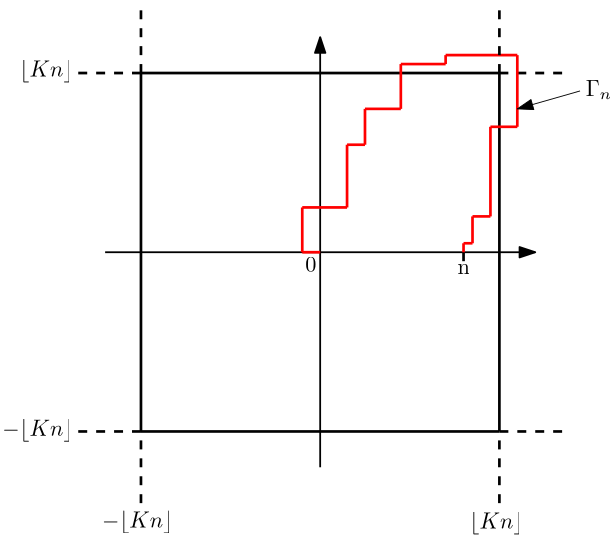}}~
\caption{The construction of $\Gamma_n$ going out of $\llbracket -\lfloor Kn \rfloor, \lfloor Kn \rfloor \rrbracket^2$).}
\label{Figure hauteur01}
\end{figure}

    Therefore $\{H(\Gamma_n) >Kn\}\subset
\{T(\Gamma_n)\ge Kn\varsigma_n(K)\}.$
Thus we have
\begin{align*}
\P(H(\Gamma_n) >Kn)&\le\P(T(\Gamma_n)\ge Kn\varsigma_n(K))\\
       &\le \P(T(\Gamma_n)\ge Kn\varsigma_n(K) \text{ and } \varsigma_n(K)\ge K^{-\frac{\beta+1}{2\beta}}n^{-\frac{1}{\beta}})+\P(\varsigma_n(K)< K^{-\frac{\beta+1}{2\beta}}n^{-\frac{1}{\beta}})\\
       &\le\underbrace{\P(T(\Gamma_n)\ge K^{\frac{\beta-1}{2\beta}}n^{1-\frac{1}{\beta}})}_{=:U_n}+\underbrace{\P((Kn)^{\frac{1}{\beta}}\varsigma_n(K)<K^{-\frac{\beta-1}{2\beta}})}_{=:V_n}.
\end{align*}
Now, since the passage times are bounded and using Theorem $1.7$ in \cite{moi}, there exists $M>0$ such that $\E[T(0,ne_1)]\le Mn^{1-\frac{1}{\beta}}$. We obtain using Markov's inequality
$$
U_n=\P(T(\Gamma_n)\ge K^{\frac{\beta-1}{2\beta}}n^{1-\frac{1}{\beta}})= \P(T(0,ne_1)\ge K^{\frac{\beta-1}{2\beta}}n^{1-\frac{1}{\beta}})\le\frac{M}{K^{\frac{\beta-1}{2\beta}}}\to_{K\to+\infty}0.
$$
Moreover, using Theorem \ref{weibull}, $(4Kn)^{\frac{1}{\beta}}\varsigma_n(K)$ converges in distribution to a Weibull random variable $Y$ of parameter $\beta$ and $1$. Thus
$$
V_n\to_{n\to+\infty}\P(Y<4^{\frac{1}{\beta}}K^{-\frac{\beta-1}{2\beta}})\to_{K\to+\infty}0.
$$
Hence we have
$$
\P(H_n>Kn)\le U_n+V_n\to_{n\to+\infty}\P(Y<4^{\frac{1}{\beta}}K^{-\frac{\beta-1}{2\beta}})\to_{K\to+\infty}0,
$$
which yields $\lim_{K\to +\infty}\liminf_{n\to +\infty}\P(H_n\le Kn)=1$. 
\end{proof}

\subsection{The case $a>0$}
This subsection is devoted to the proof of Theorem \ref{th:geoa}. Thus, we assume that $a>0$ and $F(t)\overset{t\to a}{\sim}(t-a)^\beta$ with $\beta>0$. 
\begin{proof}[Proof of Theorem \eqref{th:geoa}]
   \textbf{Lower bound.} Let $0<\eps<1$, and let $\Gamma_n^1$ be a path from $0$ to $ne_1$ with minimal passage time among all paths $\Gamma$ such that $H(\Gamma)\le \varepsilon n^{\frac{\beta}{\beta+d-1}} $.
    To prove the lower bound we construct a path $\Gamma_n^2$ such that $H(\Gamma_n^2)>\eps n^{\frac{\beta}{\beta+d-1}}$
    and we show that this path is asymptotically better than $\Gamma_n^1$.
(See Figure \ref{Figure hauteur a} for the construction in dimension $d=3$.) Since $\{H_n\le \eps n^{\frac{\beta}{\beta+d-1}}\}\subset\{T(\Gamma_n^2)-T(\Gamma_n^1)\ge0\}$
    it is sufficient to upper bound $\P\left(T(\Gamma_n^2)-T(\Gamma_n^1)\ge0\right)$.
    First, we define for $n\in \N,\eps>0$ and $x\in\Z^d$ 
    \begin{itemize}
        \item $\D_{n,\eps}(x):=\{y\in\Z^d:y_1=x_1 \text{ and }\lVert y-x \rVert_1\le \eps n^{\frac{\beta}{\beta+d-1}}\}$
    \item $\A_{n,\eps}(x):= \D_{n,\sqrt{\eps}}(x)\setminus \D_{n,\eps}(x)$
\item $\rho_n(\eps):=\min_{x\in \D_{n,\eps}(0)}{\tau_{x+\Z e_1}}$
\item $\sigma_n(\eps):= \min_{y\in \A_{n,\eps}(0)}{\tau_{y+\Z e_1}}$.
\end{itemize}
Thus, $\D_{n,\eps}(x)$ is a $(d-1)$-dimensional 
$\ell_1$-ball in $\Z^d$ centered at $x$ with radius $\eps n^{\frac{\beta}{\beta+d-1}}$. Moreover, $\A_{n,\eps}(x)$ is a $(d-1)$-dimensional annulus centered at $x$ with inner radius $\eps n^{\frac{\beta}{\beta+d-1}}$ and outer radius $\sqrt{\eps}n^{\frac{\beta}{\beta+d-1}}$. Finally, $\rho_n(\eps)$(resp. $\sigma_n(\eps)$) denotes the minimal passage time over integer lines parallel to $\Z e_1$ that intersect $\D_{n,\eps}(0)$ (resp. $\A_{n,\eps}(0)$ ).
    We construct $\Gamma_n^2$ as follows.  We denote by $m_n$ the vertex of
    $ \A_{n,\eps}(0)$ such that 
    $$
    \tau_{m_n+\Z e_1}=\sigma_n(\eps).
    $$
    Let $\gamma_{1,n}$ be a geodesic from $0$ to $m_n$. We define $\gamma_{2,n}$ as a straight path from $m_n$ to $m_n+ne_1$ and $\gamma_{3,n}$ as a geodesic from $m_n+ne_1$ to $ne_1$. We define $\Gamma_n^2$ as the concatenation of $\gamma_{1,n}$, $\gamma_{2,n}$ and $\gamma_{3,n}$. 
    \begin{figure}
\centerline{\includegraphics[scale=0.4]{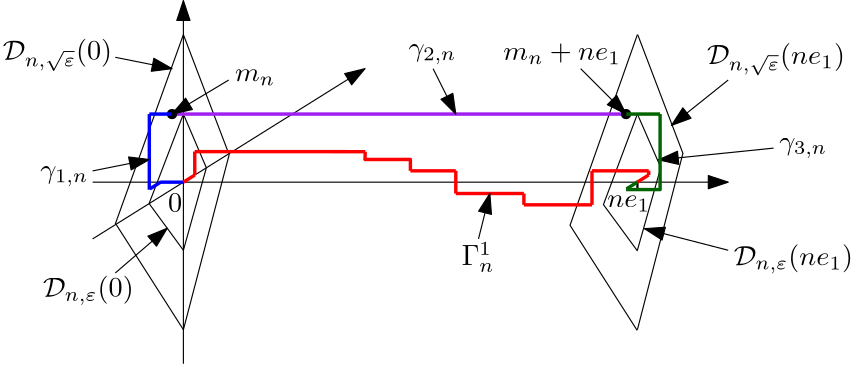}}~
\caption{The path $\Gamma_n^2$ is the concatenation of $\gamma_{1,n}$ (in blue), $\gamma_{2,n}$ (in purple) and $\gamma_{3,n}$ (in green).}
\label{Figure hauteur a}
\end{figure}
    Thus we have
    \begin{align*}
T(\Gamma_n^2)-T(\Gamma_n^1)&\le T\left( \gamma_{1,n}\right)+T\left( \gamma_{2,n}\right)+T\left( \gamma_{3,n}\right)-T\left(\Gamma_n^1\right)\\
&\le\sup_{x\in \A_{n,\eps}(0)}T(0,x)+ \sup_{x\in \A_{n,\eps}(ne_1)}T(x,ne_1)+n\left(\sigma_n(\eps)-\rho_n(\eps)\right),
\end{align*}
    which yields using the union bound
\begin{align*}
    \P\left(T(\Gamma_n^2)-T(\Gamma_n^1)\ge 0\right)&\le 2\P\left(\sup_{x\in \A_{n,\eps}(0)}T(0,x)\ge \frac{1}{2}n\left(\rho_n(\eps)-\sigma_n(\eps)\right)\right)\\
    &=2\P\left(n^{-\frac{\beta}{\beta+d-1}}\sup_{x\in \A_{n,\eps}(0)}T(0,x)\ge \frac{1}{2}n^{\frac{d-1}{\beta+d-1}}\left(\rho_n(\eps)-\sigma_n(\eps)\right)\right)\\
    &\le 2\underbrace{\P\left(n^{-\frac{\beta}{\beta+d-1}}\sup_{x\in \A_{n,\eps}(0)}T(0,x)\ge 2a\right)}_{=:I_n}\\
    &\qquad +2\underbrace{\P\left(n^{\frac{d-1}{\beta+d-1}}\left(\rho_n(\eps)-\sigma_n(\eps)\right)<4a\right)}_{=:J_n}.
\end{align*}
Now we have
$$
I_n\le \P\left(\D_{n,\sqrt{\eps}}(0)\nsubseteq B_{2n^{\frac{\beta}{\beta+d-1}}a}\right),
$$
where we recall that for $t\in\R_+$, $B_t$ is the random ball of center $0$ and radius $t$ for $T$.
Moreover, according to Theorem \ref{th:forme} applied for $\eta=\frac{1}{2}$ we have, since $\eps<1$,
\begin{align*}
\P\left(\D_{n,\sqrt{\eps}}(0)\subset B_{2n^{\frac{\beta}{\beta+d-1}}a}\text{ for all }n\text{ large enough}\right)&=\P\left(\frac{\sqrt{\eps}}{2a}\D_{1,1}(0)\subset \frac{B_{2n^{\frac{\beta}{\beta+d-1}}a}}{2n^{\frac{\beta}{\beta+d-1}}a} \;\text{ for all }n \text{ large enough}\right)\\
             &=1.
\end{align*}  
Thus $\P\left(\D_{n,\sqrt{\eps}}(0)\nsubseteq B_{2n^{\frac{\beta}{\beta+d-1}}a}\right)\to_{n\to+\infty}0$ which yields
$$
\lim_{n\to+\infty}I_n=0.
$$
Moreover, if $n\in\N$, $\mathrm{Card}(\mathcal{B}(n))\overset{n\to+\infty}{\sim}\frac{2^d}{d!}n^d$.
Therefore, for all $d\ge 2$, $\eps>0$ and $n\in \N$, if we denote $C_d:=\frac{2^d}{d!}$ we have
$$
\mathrm{Card}(\mathcal{D}_{n,\eps}(0))\overset{n\to+\infty}{\sim}C_{d-1}\eps^{d-1}n^{\frac{\beta(d-1)}{\beta+d-1}},
$$
and 
$$
\mathrm{Card}\left(\A_{n,\eps}(0)\right)\overset{n\to+\infty}{\sim}C_{d-1}(\eps^{\frac{d-1}{2}}-\eps^{d-1})n^{\frac{\beta(d-1)}{\beta+d-1}}.
$$
 Thus, according to Theorem \ref{weibull}, there exist two independent random variables $Y$ and $Y'$ following a Weibull distribution of parameters $\beta$ and $1$ such that
$$
n^{\frac{d-1}{\beta+d-1}}\left(\rho_n(\eps)-a\right)\overset{(d)}{\to}_{n\to+\infty}C_{d-1}^{-\frac{1}{\beta}}\eps^{-\frac{d-1}{\beta}}Y=:Y_\eps,
$$
and 
$$
n^{\frac{d-1}{\beta+d-1}}\left(\sigma_n(\eps)-a\right)\overset{(d)}{\to}_{n\to+\infty}C_{d-1}^{-\frac{1}{\beta}}(\eps^{\frac{d-1}{2}}-\eps^{d-1})^{-\frac{1}{\beta}}Y'=:Y'_\eps.
$$
Then, because $\rho_n(\eps)$ and $\sigma_n(\eps)$ are independent for all $n$,
$$
n^{\frac{d-1}{\beta+d-1}}\left(\rho_n(\eps)-\sigma_n(\eps)\right)\overset{(d)}{\to}_{n\to+\infty}Y_\eps-Y'_\eps.
$$
Since we have
$$
\lim_{\eps\to0}\frac{\eps^{\frac{d-1}{\beta}}}{(\eps^{\frac{d-1}{2}}-\eps^{d-1})^{\frac{1}{\beta}}}=0,
$$
then
$$
\lim_{\eps\to0}\P\left( Y_\eps-Y'_\eps<4a\right)=0.
$$
Thus
$$
\lim_{n\to+\infty}J_n=\P\left( Y_\eps-Y'_\eps<4a\right)\to_{\eps\to0}0.
$$
Hence we have
$$
\P(H_n\le \eps n^{\frac{\beta}{\beta+d-1}})\le\P\left(T(\Gamma_n^2)-T(\Gamma_n^1)\ge 0\right)\le 2(I_n+J_n)\to_{n\to+\infty}2\P\left( Y_\eps-Y'_\eps<4a\right)\to_{\eps\to0}0,
$$
which yields $\lim_{\eps\to 0}\liminf_{n\to+\infty}\P(H_n\ge \eps n^{\frac{\beta}{\beta+d-1}})=1$.\par
\medskip
\textbf{Upper bound.} Let $\Gamma_n^1$ be a path from $0$ to $ne_1$ with minimal passage time among all paths $\Gamma$ such that $H(\Gamma)>Kn^{\frac{\beta}{\beta+d-1}}$.
To prove the upper bound, we construct a path $\Gamma_n^2$ such that $H(\Gamma_n^2)\le Kn^{\frac{\beta}{\beta+d-1}}$  and show that the event $\{T(\Gamma_n^2)-T(\Gamma_n^1)\ge0\}$
has asymptotically a small probability to occur. Since $\{H_n>Kn^{\frac{\beta}{\beta+d-1}}\}\subset\{T(\Gamma_n^2)-T(\Gamma_n^1)\ge0\}$ it is sufficient to upper bound $\P(T(\Gamma_n^2)-T(\Gamma_n^1)\ge0)$.
First recall the notation $\mathcal{D}_{n,\eps}(x)$ introduced in the proof of the lower bound and define 
$$
\sigma_n(K):=\min\{\tau_{x+\Z e_1}:x\in \D_{n,\frac{K}{2}}(0)\}.
$$
 We construct $\Gamma_n^2$ as follows. Let $m_n\in \D_{n,\frac{K}{2}}(0)$ such that $\tau_{m_n+\Z e_1}=\sigma_n(K)$. Let $\gamma_{1,n}$ be a path from $0$ to $m_n$ such that $H(\gamma_{1,n})\le Kn^{\frac{\beta}{\beta+d-1}}$
with the smallest passage time. We define $\gamma_{2,n}$ as a straight path from $m_n$ to $m_n+ne_1$ and $\gamma_{3,n}$ as a path from $m_n+ne_1$ to $ne_1$ such that $H(\gamma_{3,n})\le Kn^{\frac{\beta}{\beta+d-1}}$
with the smallest passage time. We define $\Gamma_n^2$ as the concatenation of $\gamma_{1,n}$, $\gamma_{2,n}$ and $\gamma_{3,n}$. (See Figure \ref{Figure hauteur losange} for the construction in dimension $d=3$.)
\begin{figure}
\centerline{\includegraphics[scale=0.4]{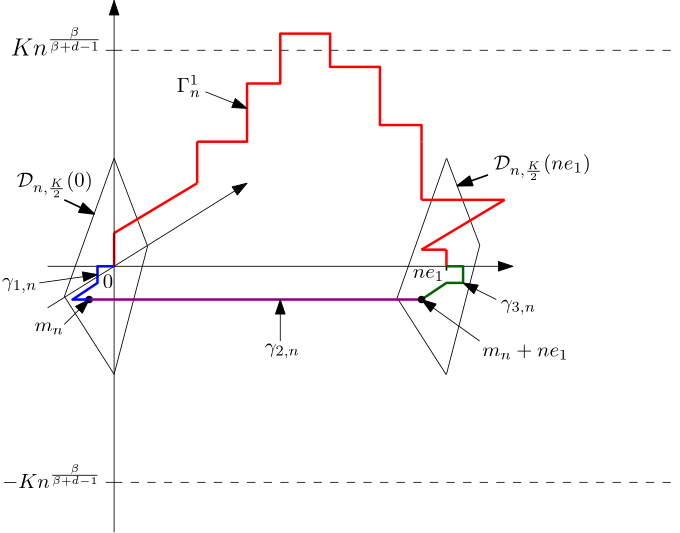}}~
\caption{The path $\Gamma_n^2$ is the concatenation of $\gamma_{1,n}$ (in blue), $\gamma_{2,n}$ (in purple) and $\gamma_{3,n}$ (in green).}
\label{Figure hauteur losange}
\end{figure}For all $x\in \Z^d$ and $y\in\D_{n,\frac{K}{2}}(x)$ we denote by
$$
\tilde{T}(x,y):=\min\{T(\Gamma), \Gamma \text{ path from }x \text{ to } y \text{ such that }H(\Gamma)\le Kn^{\frac{\beta}{\beta+d-1}}\}
$$
and
$$
\check{T}(x,y):=\min\{T(\Gamma), \Gamma \text{ path from } x \text{ to } y \text{ such that }H(\Gamma)> Kn^{\frac{\beta}{\beta+d-1}}\}
$$
Then we have
\begin{align*}
T\left(\Gamma_n^2\right)-T\left(\Gamma_n^1\right)&\le T\left(\gamma_{1,n}\right)+T\left(\gamma_{2,n}\right)+T\left(\gamma_{3,n}\right)-T\left( \Gamma_n^1\right)\\
&\le \sup_{x\in \D_{n,\frac{K}{2}}(0)} \tilde{T}\left(0,x\right) + \sup_{x\in \D_{n,\frac{K}{2}}(ne_1)}\tilde{T}\left(ne_1,x\right)+ n(\sigma_n(K)-a)-2Kn^{\frac{\beta}{\beta+d-1}}a.
\end{align*}
Thus we obtain
\begin{align*}
    \P\left(T\left(\Gamma_n^2\right)-T\left(\Gamma_n^1\right)\ge 0\right)&\le \P\left(\sup_{x\in \D_{n,\frac{K}{2}}(0)} \tilde{T}\left(0,x\right) + \sup_{x\in \D_{n,\frac{K}{2}}(ne_1)}\tilde{T}\left(ne_1,x\right)+ n(\sigma_n(K)-a)\ge 2Kn^{\frac{\beta}{\beta+d-1}}a\right)\\
    &\le 2\P\left(\underbrace{\sup_{x\in \D_{n,\frac{K}{2}}(0)} \tilde{T}\left(0,x\right)\ge \frac{2}{3}Kn^{\frac{\beta}{\beta+d-1}}a}_{=:E_n}\right)+\P\left(n\left(\sigma_n(K)-a\right)\ge\frac{2}{3}Kn^{\frac{\beta}{\beta+d-1}}a\right).
\end{align*}
Moreover, we have
\begin{align*}
    T\left(0,x\right)&=\min\left(\tilde{T}\left(0,x\right),\check{T}\left(0,x\right)\right)\\
    &\ge \min\left(\tilde{T}\left(0,x\right),Kn^{\frac{\beta}{\beta+d-1}}a\right).
\end{align*}
Thus, if $E_n$ occurs we have
$$
 \sup_{x\in \D_{n,\frac{K}{2}}(0)} T\left(0,x\right)\ge \min\left(\frac{2}{3}Kn^{\frac{\beta}{\beta+d-1}}a, Kn^{\frac{\beta}{\beta+d-1}}a\right)= \frac{2}{3}Kn^{\frac{\beta}{\beta+d-1}}a.
$$
Therefore
$$
\P\left(T\left(\Gamma_n^2\right)-T\left(\Gamma_n^1\right)\ge 0\right)\le 2\underbrace{\P\left(\sup_{x\in \D_{n,\frac{K}{2}}(0)} T\left(0,x\right)\ge \frac{2}{3}Kn^{\frac{\beta}{\beta+d-1}}a\right)}_{=:R_n}+\underbrace{\P\left(n^{-\frac{d-1}{\beta+d-1}}(\sigma_n(K)-a)\ge\frac{2}{3}Ka\right)}_{=:S_n}.
$$
 We have
$$
R_n\le \P\left(\D_{n,\frac{K}{2}}(0)\nsubseteq B_{\frac{2}{3}Kn^{\frac{\beta}{\beta+d-1}}a}\right).
$$
Now, according to Theorem \ref{th:forme} applied for $\eta=\frac{1}{4}$ we have
\begin{align*}
\P\left(\D_{n,\frac{K}{2}}(0)\subset B_{\frac{2}{3}Kn^{\frac{\beta}{\beta+d-1}}a}\text{ for all }n\text{ large enough}\right)&=\P\left(\frac{3}{4a}\D_{1,1}(0)\subset \frac{B_{\frac{2}{3}Kn^{\frac{\beta}{\beta+d-1}}a}}{\frac{2}{3}Kn^{\frac{\beta}{\beta+d-1}}a} \;\text{ for all }n \text{ large enough}\right)\\
             &=1.
\end{align*}
This yields that $\lim_{n\to+\infty}\P\left(\D_{n,\frac{K}{2}}(0)\nsubseteq B_{\frac{2}{3}Kn^{\frac{\beta}{\beta+d-1}}a}\right)=0$. Hence 
\begin{equation}\label{eq:rn}
\lim_{n\to+\infty}R_n=0.
\end{equation}
Moreover, for all $d\ge 2$, if we denote $C_d:=\frac{2^d}{d!}$ we have
$$
\mathrm{Card}\left(\D_{n,\frac{K}{2}}(0)\right)\overset{n\to+\infty}{\sim}C_{d-1} \left(\frac{K}{2}\right)^{d-1}n^{\frac{\beta(d-1)}{\beta+d-1}}.
$$
Thus, according to Theorem \ref{weibull} there exists a random variable $Y$ following a Weibull distribution of parameters $\beta$ and $1$ such that
$$
n^{\frac{d-1}{\beta+d-1}}(\sigma_n(K)-a)\overset{(d)}{\to}_{n\to+\infty}C_{d-1}^{-\frac{1}{\beta}} \left(\frac{K}{2}\right)^{-\frac{d-1}{\beta}}Y=:Y_K\overset{(d)}{\le} Y_1,
$$
where the inequality is true for $K>1$, and 
$$
\lim_{K\to+\infty}\P\left( Y_1\ge \frac{2}{3}Ka\right)=0,
$$
since $Y_1<+\infty$ almost surely.
Thus for $K>1$
\begin{equation}\label{eq:sn}
\lim_{n\to+\infty}S_n=\P\left(Y_K\ge \frac{2}{3}Ka \right)\le\P\left(Y_1 \ge \frac{2}{3}Ka\right)\to_{K\to+\infty}0.
\end{equation}
Therefore we have
$$
\limsup_{n\to+\infty}\P(H_n>Kn^{\frac{\beta}{\beta+d-1}})\le \limsup_{n\to+\infty}\P\left(T\left(\Gamma_n^2\right)-T\left(\Gamma_n^1\right)\ge 0\right)\le \limsup_{n\to+\infty}(2R_n+S_n),
$$
which yields thanks to \eqref{eq:rn} and \eqref{eq:sn}
$$
\limsup_{n\to+\infty}\P(H_n>Kn^{\frac{\beta}{\beta+d-1}})\le\P\left(Y_K\ge \frac{2}{3}Ka \right)\to_{K\to+\infty}0.
$$
Hence
$$
\lim_{K\to+\infty}\liminf_{n\to+\infty}\P(H_n\le Kn^{\frac{\beta}{\beta+d-1}})=1.
$$.
\end{proof}

\section*{Acknowledgments}
I would like to warmly thank my two PhD advisors, Anne-Laure Basdevant and Lucas Gerin, as well as the ANR LOUCCOUM project which supported this work.

\bibliographystyle{alpha}
\bibliography{biblio}
\end{document}